\documentclass[12pt]{amsart}

\usepackage{amsmath,amssymb,amsfonts,amsthm,amscd,indentfirst}

\usepackage{graphicx}

\usepackage{indentfirst}
\usepackage{hyperref}

\usepackage{color}

\usepackage{setspace}
\usepackage{amsmath, amssymb,amsthm, amsfonts, color}
\newtheorem{thm}{Theorem}

\theoremstyle{definition}

\theoremstyle{remark}

\newcommand{\R}{\mathbb R}

\newcommand{\be}{\begin{equation}}
\newcommand{\ee}{\end{equation}}
\newcommand{\bee}{\begin{equation*}}
\newcommand{\eee}{\end{equation*}}

\def\p{\partial}

\def\Pi{\displaystyle{\mathbb{II}}}

\begin{document}

\title{Nonexistence of NNSC fill-ins with large \\ 
mean curvature}

\author{Pengzi Miao}
\address[Pengzi Miao]{Department of Mathematics, University of Miami, Coral Gables, FL 33146, USA}
\email{pengzim@math.miami.edu}

\thanks{The author's research was partially supported by NSF grant DMS-1906423.}


\begin{abstract}
In this note we show that a closed Riemannian manifold does not admit a fill-in with nonnegative scalar curvature
if the mean curvature is point-wise large.  
Similar result also holds for fill-ins with a negative scalar curvature lower bound.
\end{abstract}


\maketitle

\markboth{Pengzi Miao}{Nonexistence of NNSC fill-ins}

Consider a $2$-sphere $S^2$ with the standard round metric $\gamma_o$ of area $4\pi$.
If $H$ is a function on $S^2$ with $ H > 2$, there does not exist any compact Riemannian $3$-manifold 
$(\Omega, g)$ with nonnegative scalar curvature, with boundary such that $\p \Omega$ is isometric to $(S^2, \gamma_o)$ and
 has mean curvature  $H$. This can be derived as 
 a consequence of the Riemannian positive mass theorem \cite{SchoenYau79, Witten81}, 
 formulated 
 on manifolds with corner along hypersurfaces \cite{Miao02, ShiTam02}. 
 
For an arbitrary closed orientable surface $\Sigma$, a pair $(\gamma, H)$ is called Bartnik data on $\Sigma$ if $\gamma$ and $H$ 
denote  a metric and a function on $\Sigma$, respectively. The question whether $(\Sigma, \gamma, H)$ bounds a compact $3$-manifold
with nonnegative scalar curvature, with boundary  isometric to $(\Sigma, \gamma)$ and having mean curvature 
 $H$ is closely tied to the quasi-local mass problem of $(\Sigma, \gamma, H)$ (see \cite{Bartnik89, Jauregui13} for instance).

In general, let $\Sigma^{n-1}$ be a closed  $(n-1)$-dimensional manifold. Given a metric $\gamma$ and 
a function $H$ on $\Sigma$, 
we say a compact orientable  Riemannian $3$-manifold $(\Omega, g)$ is a nonnegative scalar curvature (NNSC) fill-in of $(\Sigma, \gamma, H)$ 
if  $\p \Omega$ is isometric to $(\Sigma, \gamma)$ and the mean curvature of $\p \Omega$  equals $H$. 
In \cite{Gromov18}, Gromov showed, if $\Omega$ is a spin manifold and if $(\Omega^n, g) $ is 
a NNSC fill-in of $(\Sigma, \gamma, H)$, then 
\be \label{eq-Gromov-1}
\min_\Sigma H \le \frac{n-1}{ \mathrm{Rad} (\Sigma,\gamma) },
\ee
where $\text{Rad} (\Sigma, \gamma) $ is a constant only depending on $(\Sigma, \gamma) $, known as the (hyper)spherical radius of 
$(\Sigma, \gamma)$. As a result, spin NNSC fill-ins of $(\Sigma, \gamma, H)$ do not exist for large mean curvature function $H$. 

One can drop the requirement on $H$ when studying the geometry of  fill-ins. We say $(\Omega, g)$ is a fill-in of $(\Sigma, \gamma)$
if $\p \Omega$ is isometric to $(\Sigma, \gamma)$. Interaction among the scalar curvature, the total boundary 
mean curvature, and the volume of fill-ins were studied  in \cite{ShiTam02, JMT13, Mantoulidis-Miao16, MiaoTam08, SWWZ, SWW20}. 

In \cite{Gromov19}, Gromov raised the following  existence question of fill-ins with positive scalar curvature. 

\vspace{.2cm}

\noindent {\bf Question 1} (\cite{Gromov19}) 
{\sl 
If $ \Sigma = \p \Omega$ for a compact manifold $\Omega$ and $\gamma$ is a Riemannian metric on $\Sigma$, 
does $\gamma$ extend to a Riemannian metric $g$ on $\Omega$ with positive scalar curvature?
}

\vspace{.2cm} 

This question recently has been answered affirmatively by Shi, Wang and Wei  \cite{SWW20}.

\begin{thm} [Shi-Wang-Wei \cite{SWW20}] \label{thm-SWW}
Let $\Omega^n$ be a compact $n$-dimensional manifold with boundary $\Sigma $. 
Then any metric $\gamma$ on $\Sigma$ can be extended to a Riemannian metric $g$ on $\Omega$ with positive scalar curvature.
\end{thm}

To construct such an extension, the authors made an ingenious use of the parabolic method employed in \cite{Bartnik93, ShiTam02}.
More precisely, they started with a metric $\bar g$ on $\Omega$ with positive scalar curvature
(whose existence is guaranteed by \cite{Gromov69, KW75} for instance), 
and construct a suitable transition metric on   $\Sigma \times [0, 1]$, which connects $\gamma$ on $\Sigma \times \{ 0 \}$ 
to a (large) constant  scaling of the induced metric from $\bar g$  on $ \Sigma = \Sigma \times \{ 1 \}$,  via the parabolic method.

Applying their proof of Theorem \ref{thm-SWW} and the positive mass theorem, 
Shi, Wang and Wei \cite{SWW20} further proved the following nonexistence theorem on NNSC fill-ins.
Hereinafter, the dimension $n$ denotes a dimension for which the Riemannian positive mass theorem holds. (See
the recent work of Schoen-Yau \cite{SchoenYau17} and references therein.)

\begin{thm}[Shi-Wang-Wei \cite{SWW20}] \label{thm-SWW-class-c}
Suppose a closed manifold $\Sigma^{n-1}$ can be topologically embedded in $\R^n$. Given any Riemannian metric $\gamma$ 
on $\Sigma$, there exists a constant $h_0$, depending on $\Sigma$, $\gamma$, and the embedding of $\Sigma$ in $ \R^n$, 
such that, if $ \min_\Sigma H \ge h_0 $,  there do not exist  NNSC fill-ins of $(\Sigma, \gamma, H )$.
\end{thm}

It is the purpose of this note to show that no NNSC fill-ins exist for any  $\Sigma$, if $H$ is large.
The proof makes use of Shi-Wang-Wei's Theorem \ref{thm-SWW} above and Schoen-Yau's result on closed manifolds 
with positive scalar curvature \cite{SchoenYau79-MM}.

\begin{thm} \label{thm-main}
Let $ \Sigma^{n-1} $ be the boundary of some compact $n$-dimensional manifold $\Omega$.
Given any Riemannian metric $\gamma$ on $\Sigma$,
there exists a constant $H_0 $, depending on $\gamma$ and $\Omega$, such that,
if $ \min_\Sigma H \ge H_0  $,  there do not exist  NNSC fill-ins of $(\Sigma, \gamma, H )$.
\end{thm}

\begin{proof}
Let $p$ be  an interior point  in $\Omega$. Near $p$, form a connected sum of $\Omega$ with $T^n$, where
$T^n$ is the $n$-dimensional torus. Denote the resulting manifold by $\widetilde \Omega = \Omega \# T^n$, 
then $\p \widetilde \Omega = \p \Omega$.

Given the metric $\gamma$ on $\Sigma$, apply Theorem \ref{thm-SWW} to $\Sigma = \p \widetilde \Omega  $,
one obtains a metric $ g $ with positive scalar curvature on $\widetilde \Omega  $ such that  $g$ induces the metric 
$\gamma$ on $\Sigma $. 
Let $H_{_{\widetilde \Omega}} $ be the mean curvature of $\Sigma$ in $(\widetilde \Omega, g)$ with respect to the inward unit normal. 

Now suppose  $(M, g_{_M})$ is a compact manifold with nonnegative scalar curvature so that $\p M$ is isometric to
$(\Sigma, \gamma)$.  Let $H_{_M}$ be the mean curvature of $\Sigma = \p M $ in $(M, g_{_M})$ with respect to the outward unit normal.
Suppose
\be \label{eq-mean-curvature-H}
\min_\Sigma H_{_M} \ge \max_\Sigma H_{_{\widetilde \Omega}}. 
\ee

Consider the manifold $( \widetilde M , \tilde g)$ obtained by gluing $(M, g_{_M})$ and $(\widetilde \Omega, g)$ along their common boundary $(\Sigma, \gamma)$.
The metric $\tilde g$ has nonnegative scalar curvature in $M$, has positive scalar curvature in $\tilde \Omega$, and satisfies 
\eqref{eq-mean-curvature-H} across $\Sigma $ in $ \widetilde M$. By the interpretation of \eqref{eq-mean-curvature-H} 
as the metric having nonnegative distributional scalar curvature across $\Sigma$ \cite{Miao02}, 
one expects $\tilde g$ can be mollified to produce a smooth positive scalar curvature metric on $\widetilde M$.

This expectation can be verified via results in \cite{ShiTam16} (also see \cite{LiMantoulidis18}) for instance. 
By Corollary 4.2 in \cite{ShiTam16}, there exists a sequence of smooth metrics $\{ g_i \}$ on $ \widetilde M$ such that 
$ g_i$ has nonnegative scalar curvature
and $g_i$ converges  to $\tilde g$ in $C^\infty$ norm on compact sets away from $\Sigma$.
Since $\tilde g = g $ has positive scalar curvature on $\tilde \Omega$, this implies  $g_i$ has positive scalar curvature 
somewhere in $\widetilde \Omega$. Consequently, $\widetilde  M $ supports a metric with positive scalar curvature. 

However, by construction,  $\widetilde M$ has topology 
\be
\widetilde M = K \# T^n , 
\ee
where $K$ is an $n$-dimensional closed orientable manifold obtained by gluing $M $ and $\Omega$ along their common boundary $\Sigma$.
By Schoen-Yau's result on closed manifolds \cite{SchoenYau79-MM} (also see \cite{SchoenYau17}), $\widetilde M$ does not admit a metric with positive scalar curvature. 

This gives a contradiction to \eqref{eq-mean-curvature-H}.
The claim follows by taking  $ H_0 = \max_\Sigma H_{_{\widetilde \Omega} }$.
\end{proof}

Combined with a trick of Gromov \cite{Gromov18},  Theorem \ref{thm-main} implies 
a similar result for fill-ins with a negative scalar curvature lower bound.

\begin{thm} \label{thm-main-n}
Let $ \Sigma^{n-1} $ be the boundary of some compact $n$-dimensional manifold $\Omega$.
Let $ \sigma > 0 $ be a constant. Given a Riemannian metric $\gamma$ on $\Sigma$,
there exists a constant $ H_\sigma $, depending on $\gamma$, $\Omega$ and $\sigma$, such that,
if $ \min_\Sigma H \ge H_\sigma  $,  $(\Sigma, \gamma, H )$ does not have fill-ins with scalar curvature 
bounded below by $ - \sigma$.
\end{thm}

\begin{proof}
Let $({S}^{m} (r) , g_{_S})$ denote a standard $m$-dimensional round sphere with radius $ r$. 
Here $ m \ge 2$ and $ r$ is chosen so that $ m ( m -1) = r^2 \sigma $.

Suppose $(M, g_{_M})$ is a fill-in of $(\Sigma, \gamma)$ with $ R (g_{_M} ) \ge - \sigma $.
Following \cite{Gromov18}, consider the Riemannian product $ (N^{n+m}, g_{_N} ) =  (M, g_{_M}) \times (S^{m} (r) , g_{_S}) $.
Then $ R ( g_{_N} ) \ge 0 $, and $ \p N = \Sigma \times S^{m} (r) $ has 
mean curvature $H_{_M}$ in $(N, g_{_N})$. 
The claim follows by taking $H_\sigma = H_0$, where $H_0$ is the constant given in Theorem \ref{thm-main} when 
it is applied to $ \Sigma \times S^m (r) = \p ( \Omega^n \times S^m (r) ) $
with the metric  $\gamma + g_{_S}$.
\end{proof}

\vspace{.1cm}

We conclude this note with some questions open to the author's knowledge. 
Given a pair $(\Sigma, \gamma)$ with $\Sigma $ being the boundary of some compact manifold, 
let $\mathcal{F}_{{(\Sigma, \gamma)} }$ 
denote the set of NNSC fill-ins of $(\Sigma, \gamma)$. 
Shi-Wang-Wei's extension theorem shows
$
 \mathcal{F}_{{ (\Sigma, \gamma) } } \neq \emptyset . 
$
However, the fill-in produced in \cite{SWW20} has a feature that 
its boundary $\Sigma$ has negative mean curvature, i.e. 
the mean curvature vector of $\Sigma$ points outward. 
This leads to a question similar with but different to Question 1:

\vspace{.2cm}

\noindent {\bf Question 2}.
{\sl 
If $ \Sigma $ is the boundary of some compact manifold  and $\gamma$ is a Riemannian metric on $\Sigma$, 
is $\mathcal{F}^{ \, +} _{{(\Sigma, \gamma)} } \ne \emptyset$? Here 
$  \mathcal{F}^+_{{ (\Sigma, \gamma) } } $ denotes the set of NNSC fill-ins of $(\Sigma, \gamma)$ 
so that $\Sigma$ has positive mean curvature. 
}

\vspace{.2cm}

From the point of view of quasi-local mass \cite{ShiTam02, WY09}, the topology of fill-ins are allowed to vary. 
If the topology of fill-ins is fixed to be some $\Omega$, a recent result of  
Carlotto and Li \cite{Carlotto-Li19} completely determines the topology of these $\Omega$
in $3$-dimension, assuming $\Omega$ admits a positive scalar curvature metric with mean convex boundary.

In terms of the set $\mathcal{F}_{(\Sigma, \gamma)}$, Theorem \ref{thm-main} shows 
\be \label{eq-sup-min-H}
\sup_{(M, g_{_M} ) \in \mathcal{F}_{(\Sigma, \gamma)} } \min_\Sigma H_{_M} < \infty . 
\ee
Considering Gromov's result \eqref{eq-Gromov-1}, it would be interesting to know  
if the left side of \eqref{eq-sup-min-H}
can be estimated  explicitly by a metric quantity of  $(\Sigma, \gamma)$?

\vspace{.3cm}

\noindent {\bf Acknowledgements.}
I thank Yuguang Shi for helpful discussions in relation to Theorem \ref{thm-main}.

\end{document}